\documentclass{amsart}
\usepackage{amsrefs}
\title{On the motivic Peterson conjecture}
\author{Masaki Kameko}
\address{Department of Mathematical Sciences,
Shibaura Institute of Technology, 
307 Minuma-ku Fukasaku, Saitama-City 337-8570, Japan}
\thanks{This work was supported by JSPS KAKENHI Grant Numbers JP25400097 and JP17K05263.}
\email{kameko@shibaura-it.ac.jp}
\newtheorem{theorem}{Theorem}[section]
\newtheorem{proposition}[theorem]{Proposition}
\newtheorem{lemma}[theorem]{Lemma}
\newtheorem{conjecture}[theorem]{Conjecture}
\theoremstyle{definition}

\newtheorem{remark}[theorem]{Remark}

\begin{document}

\begin{abstract}
We show that the analogue of the Peterson conjecture on the action of Steenrod squares does not hold
in motivic cohomology.
\end{abstract}

\maketitle

\newcommand{\sq}{\mathrm{Sq}}

\section{Introduction}\label{section:1}


The mod $2$ cohomology $H^{*}(BV_n)$ of the classifying space of the elementary abelian $2$-group 
$V_n=(\mathbb{Z}/2)^n$ of rank $n$ is a polynomial algebra of $n$-variables
\[
\mathbb{Z}/2[x_1, \dots, x_n],
\]
where $\deg  x_i=1$. We denote it by $P_n$. Also, we denote by $P_n^d$ the subspace of $P_n$ spanned by monomials of degree $d$. Then, we have
\[
P_n=\bigoplus_{d=0}^{\infty} P_n^d. 
\]
Throughout  this paper,   an element
means a homogeneous element.
 The action of the Steenrod squares 
$\sq^1$, $\sq^2$, $\sq^3, \dots$ on $P_n$  is given by the unstable condition
\begin{align*}
\sq^1 (x_i)&=x_i^2,\\
 \sq^a (x_i)&=0 \quad \mbox{for $a>1$},
\end{align*}
and the Cartan formula 
\[
\sq^{a} (x\cdot y)= \sum_{b=0}^a \sq^{a-b} (x)\cdot \sq^b (y), 
\]
where $x, y \in P_n$ and   $\sq^0$ is the identity map, that is, $\sq^{0}(x)=x$, $\sq^0 (y)=y$.

Finding  a minimal set of generators of $P_n$ as a module over the mod $2$ Steenrod algebra 
$\mathcal{A}$ is known as the hit problem. One may consider the quotient space
\[
QP_n^d=P_n^{d} /(\mathcal{A}_{+}(P_n)\cap P_n^d),
\]
where
$\mathcal{A}_{+}$ is the set of positive degree elements in the mod $2$ Steenrod algebra $\mathcal{A}$ and 
\[
\mathcal{A}_{+}(P_n)=\{ a x  \;|\; a \in \mathcal{A}_{+}, \; x \in P_n \}.
\] 
Then, the hit problem is the problem of finding a basis for $QP_n^d$.
Since its formulation in mid-1980 by Peterson through the computation of $QP_2^d$,
the hit problem has been and is studied by many mathematicians. Among recently published papers and books  are 
Ault \cite{ault-2014}, Pengelley and Williams \cite{pengelley-williams-2015}, and 
Sum  \cite{sum-2015} and Walker and Wood  \cite{walker-wood-2018}.


The most fundamental result on the hit problem is Wood's theorem. It was known as 
Peterson's conjecture before Wood proved it.  
The Peterson conjecture was formulated in terms of the number $\alpha(d+n)$ of $1$'s in the binary expression of $d+n$. 
However, we may state it in terms of the function $\beta(d)$ defined by
\[
\beta(d)=\min\{\;  s \in \mathbb{N} \; |\; d=(2^{i_1}-1)+\cdots + (2^{i_s}-1), \; i_1, \dots, i_s \in \mathbb{N} \; \}.
\]


\begin{theorem}[Wood \cite{wood-1989}] \label{theorem:1.1}
 If  $\beta(d)>n$, then 
$
 \dim QP_n^d=0.
$
\end{theorem}

Wood's theorem gives us a  sufficient condition for $QP_n^d=\{0\}$
 in terms of $d$ and $n$.  The condition $\beta(d)>n$ is also a necessary condition for $QP_n^d=\{0\}$
 since 
a monomial of the form 
 \[
 x_1^{2^{i_1}-1} \cdots x_{n}^{2^{i_n}-1}
 \]
 is not in $\mathcal{A}_{+}(P_n)$.
Furthermore, Wood's theorem is the foundation for various  results on the hit problem, for example, the computation of $\dim QP_3^d$ and so on. 
Singer's transfer homomorphism \cite{singer-1989}  relates  the hit problem to the cohomology of the mod $2$ Steenrod algebra. 
 It is the $E_2$-term of the classical  Adams spectral sequence hence the hit problem is related to 
 the stable homotopy theory. 
 Notably, Minami's new doomsday conjecture \cite{minami-1995} is 
 inspired by the Peterson conjecture and its consequences.


On the other hand, in the $21^{\rm st}$ century,  motivic cohomology theory is 
studied  in both algebraic geometry and algebraic topology. In particular, a motivic analogue
of the Adams spectral sequence and its $E_2$-term, that is, 
the cohomology of the mod $2$ motivic Steenrod algebra,  
is studied by Dugger and Isaksen \cite{dugger-isaksen-2010}. So, it is reasonable to 
think of motivic counterparts of the hit problem, Singer transfer, new doomsday conjecture and so on.


In this paper, we consider the hit problem in motivic cohomology and an analogue of the Peterson conjecture. 
To be precise, we disprove the motivic version of the Peterson conjecture. 
Our result seems to indicate a significant  difference between the classical stable homotopy theory and 
the motivic stable homotopy theory. 
We hope our result sheds some light on both classical and motivic stable homotopy theory.


For simplicity, we assume that the base field is the complex number field. Then, 
$H^{*,*}(\mathrm{Spec}(\mathbb{C}))=\mathbb{Z}/2[\tau]$, where $\deg \tau=(0,1)$. 
We refer the reader to \cite[Section 9] {voevodsky-2003} and \cite[Section 2]{yagita-2005}
for the details of the mod $2$ motivic Steenrod algebra $\mathcal{A}^{*,*}$ and the mod $2$ motivic cohomology $H^{*,*}(BV_n)$ of $BV_n$.
The mod $2$ motivic cohomology of the classifying space  of elementary abelian $2$-group $V_n$ of rank $n$  is  given by
\[
M_n=\mathbb{Z}/2[\tau, x_1, \dots, x_n, y_1, \dots, y_n]/
(x_1^2+\tau y_1, \dots, x_n^2+\tau y_n), 
\]
where
$ \deg \tau= (0,1)$, $\deg x_i=(1,1)$, $\deg y_i=(2, 1).
$
The mod $2$ motivic Steenrod algebra is generated by $Q_0$, $\wp^1, \wp^2, \wp^3, \dots$. Its action on the 
$M_n$ is given by the unstable condition
\begin{align*}
Q_0(\tau)&=0,  &\wp^a(\tau)&=0 \quad \mbox{for $a\geq 1$,} 
\\
Q_0(x_i)&=y_i, & \wp^{a}(x_i)&=0 \quad \mbox{for $a\geq 1$,} 
\\
Q_0(y_i)&=0 &  \wp^{1}(y_i)&=y_i^2, 
&
\wp^{a}(y_i)&=0 \quad \mbox{for $a\geq 2$,}
\end{align*}
and the Cartan formula
\begin{align*}
Q_0(x y)&= Q_0(x )y + x Q_0(y), 
\\
 \wp^a (xy)&= \sum_{b=0}^{a} \wp^{a-b}(x) \wp^{b}(y)+\tau  \sum_{b=0}^{a-1} 
Q_0\wp^{a-1-b} (x) Q_0\wp^{b}(y). 
\end{align*}
The mod $2$ Steenrod algebra is given as the quotient of the mod $2$ motivic Steenrod algebra, that is, 
$\mathcal{A}=\mathcal{A}^{*,*}/(\tau+1)$. 
Similarly, we have $M_n/(\tau+1)=P_n$ and  the projection 
\[
M_n \to P_n, 
\]
sending $\tau$, $x_i$, $y_i$  to $1$, $x_i$,  $x_i^2$, respectively. 
This projection  is nothing but the realization map 
\[
H^{*,*}(BV_n)\to H^{*}(BV_n).
\]
We denote  by 
$\mathcal{A}^{*,*}_+$
the set of elements $a$ in $\mathcal{A}^{*,*}$ such that the sum of the  first and second  degrees of $a$ is positive
and let 
\[
\mathcal{A}_{+}^{*,*}(M_n)=\{ a x \;|\; a \in \mathcal{A}_{+}^{*,*}, \; x \in M_n \}.
\]
Let $M_n^{d, *}$ be the subspace of $M_n$ spanned by elements of degree $(d, *)$.
We define $QM_n^{d,*}$ by
\[
QM_n^{d,*}=M_n^{d,*}/(\mathcal{A}^{*,*}_{+}(M_n)\cap M_n^{d,*}).
\]
We call the problem of finding a minimal set of generators of $M_n$ as a module over the mod $2$ motivic 
Steenrod algebra $\mathcal{A}^{*,*}$ the motivic hit problem.
It is equivalent to the problem of finding a basis for $QM_n^{d,*}$. 
A monomial  of the form
\[
x_1\cdots x_n y_1^{2^{i_1-1}-1}\cdots y_n^{2^{i_n-1}-1}
\]
 in $M_n^{d,*}$ 
 is not in $\mathcal{A}^{*,*}_+(M_n)$ . So, 
if $\beta(d)\leq n$,  $\dim QM_n^{d, *}\not= 0$. 
Thus, it is reasonable to ask the following conjecture.


\begin{conjecture}\label{conjecture:1.2}
If  $\beta(d)>n$, then
$
\dim Q M_n^{d,*}=0.
$
\end{conjecture}

This conjecture holds for some $n$ and 
this is the motivic  analogue of the Peterson conjecture. The purpose of this paper is to give 
counterexamples for this conjecture.


\begin{theorem} \label{theorem:1.3}
For $n$ such that $\alpha(n-2)\geq 3$, 
let $k=n-3$, $d=(n-1)(2^{k+1}-2)+k$. Then, $
\beta(d)>n$ but 
$
\dim QM_n^{d,*}\not=0.
$
\end{theorem}

For $n=9$, the assumption of Theorem~\ref{theorem:1.3} holds.  For $n=9$, we have $k=6$, $d=1014$
and $\beta(1014)=10>9$.
We prove Theorem~\ref{theorem:1.3} by giving a family of monomials not in $\mathcal{A}^{*,*}_{+}(M_n)$.
Let us denote by $z_k$ the monomial
\[
x_1\cdots x_k y_1^{2^{k-1}-1}y_2^{2^k-2^{k-2}-1} \cdots y_k^{2^{k}-2^0-1} y_{k+1}^{2^{k}-1}\cdots y_n^{2^{k}-1}
\]
of degree $d=(n-1)(2^{k+1}-2)+k$.
Then, Theorem~\ref{theorem:1.3} could be divided to the following 
Propositions~\ref{proposition:1.4} and \ref{proposition:1.5}.


\begin{proposition}\label{proposition:1.4}
Suppose that $1\leq k <n$. Then, the monomial 
$
z_k 
$
is not in $\mathcal{A}^{*,*}_{+}(M_n)$.
\end{proposition}


\begin{proposition}\label{proposition:1.5}
Suppose that $\alpha(n-2)\geq 3$ and $k=n-3$. Then, we have
$
\beta(d)>n.
$
\end{proposition}

This paper is organized as follows. In Section~\ref{section:2}, we prove Proposition~\ref{proposition:1.5}.
In Section~\ref{section:3}, we recall some results on the classical hit problem.
In Section~\ref{section:4}, we 
give the details of  the motivic hit problem.
In Section~\ref{section:5}, we prove Proposition~\ref{proposition:1.4}.

The author would like to thank Nobuaki Yagita and the referee for comments and suggestions
for improving the exposition of this paper.


\section{Proof of Proposition~\ref{proposition:1.5}}\label{section:2}

In  this section, we prove Proposition~\ref{proposition:1.5}. First, we prove Proposition~\ref{proposition:2.1} below.


\begin{proposition} \label{proposition:2.1}
Let $d$, $n$ be positive integers.
Then, $\alpha(d+n)>n$ if and only if  $\beta(d)>n$, where $\alpha(d+n)$ is the number of $1$'s in the binary expansion of $d+n$.
\end{proposition}

\begin{proof}
We prove that $\alpha(d+n)\leq n$ if and only if $\beta(d)\leq n$.
Recall that 
\[
\alpha(t)=\min \{ \; s \;|\; t=2^{i_1}+\cdots + 2^{i_s}, \;  i_1\geq i_2\geq \cdots \geq i_s \geq 0\;   \}.
\]
Therefore, we have  $\alpha(s)\leq s$.

Suppose that $\beta(d)\leq n$. Then, we have
\[
d=(2^{i_1}-1)+\cdots +(2^{i_n}-1)
\]
for some $ i_1\geq \cdots \geq i_n\geq 0$.
Hence, we have
\[
d+n=2^{i_1}+\cdots +2^{i_n}.
\]
Therefore, we have $\alpha(d+n)\leq n$.

Suppose that $\alpha(d+n)\leq n$.
Let $s$ be the least positive integer such that 
\[
\alpha(d+s)\leq s.
\]
By the definition of $s$, we have $s\leq n$.
Let $r=\alpha(d+s)$. Then, again, by the definition of $s$, $r\leq s$
and
\[
d+s=2^{i_1}+\cdots + 2^{i_r}
\]
for some $ i_1>\cdots > i_r \geq 0$. 
So, 
\[
d+s-1=2^{i_1}+\cdots + 2^{i_r}-1=2^{i_1}+\cdots + 2^{i_{r-1}} + 2^{i_r-1}+\cdots +2^0
\]
and by the definition of $s$, we have
$\alpha(d+s-1)=r-1+i_r>s-1$. Therefore, we have $i_r>s-r\geq 0$. 
Let $j_a=i_a$ for $a\leq r-1$, $j_a=i_{r}-1-a+r$ for $r\leq a \leq s-1$ and $j_a=i_r-s+r$ for $a=s$.
Then, 
\[
d+s=2^{j_1}+\cdots+ 2^{j_s}.
\]
Hence, we have
\[
d=(2^{j_1}-1)+\cdots+(2^{j_s}-1).
\]
Therefore, we have $\beta(d)\leq s \leq n$. 
\end{proof}


\begin{remark}\label{remark:2.2}
Suppose that $d=(n-1)(2^{k+1}-2)+k$. Then, we have
\[
d+n=(n-1)\cdot 2^{k+1}+2-n+k.
\]
For $k=n-1$, we have
\[
d+n=(n-1)\cdot 2^{n}+1.
\]
Hence, we have 
\[
\alpha(d+n)=\alpha(n-1)+1\leq n.
\]
Similarly, for $k=n-2$, we have
\[
d+n=(n-1) \cdot 2^{n-1}
\]
and 
\[
\alpha(d+n)=\alpha(n-1)\leq n.
\]
Therefore, $\beta(d)\leq n$ for $k=n-1, n-2$. 
\end{remark}

Now, we prove Proposition~\ref{proposition:1.5}.


\begin{proof}[Proof of Proposition~\ref{proposition:1.5}]
Let $d=(n-1)(2^{k+1}-2)+k$ and $k=n-3$. Then, 
\[
d+n=(n-1)\cdot 2^{n-2}-1=(n-2)\cdot 2^{n-2}+ 2^{n-3}+\cdots +2^0.
\]
Hence, we have
\[
\alpha(d+n)=\alpha(n-2)+ n-2.
\]
Since we assumed that $\alpha(n-2)\geq 3$, we have 
\[
\alpha(d+n)>n.
\]
Thus, by Proposition~\ref{proposition:2.1}, we have $\beta(d)>n$.
\end{proof}

\section{The classical hit problem}\label{section:3}

In this section, we recall some results on the classical hit problem. 

For a monomial 
\[
v=x_1^{e_1}\cdots x_n^{e_n}
\]
in $P_n$, 
let us define  $\alpha_{ij}(v)$ in $\{0,1\}$ by
\[
e_i=\sum_{j=0}^\infty \alpha_{ij}(v) 2^j.
\]
We define non-negative integers $\alpha_i(v)$, $\omega_j(v)$ by 
\begin{align*}
\alpha_i(v)&= \sum_{j=0}^{\infty} \alpha_{ij}(v), 
\\
\omega_j(v)&=\sum_{i=1}^n \alpha_{ij}(v).
\end{align*}
For  finite sequences of non-negative integers  of the same length $c$, say
\[
\gamma=(\gamma_1, \dots, \gamma_c), \quad \mbox{and} \quad \delta=(\delta_1, \dots, \delta_c),
\]
we consider the lexicographic order from the left, that is, 
we say 
\[
\gamma<\delta
\]
if and only if there exists $a\geq 1$ such that $\gamma_b=\delta_b$ for $b<a$ and $\gamma_a<\delta_a$.

Let $F_n$ be the subspace of $P_n$ spanned by $\mathcal{A}_{+}(P_n)$ and monomials $v$ in $P_n$  such that 
\[
(\omega_0(v), \dots, \omega_{k-1}(v))<(n-1, \dots, n-1).
\]
We denote by $F_{n-1}$ the subspace of $P_{n-1}$ spanned by $\mathcal{A}_{+}(P_{n-1})$ and 
monomials $v$ in $P_{n-1}$ such that 
\[
(\omega_0(v), \dots, \omega_{k-1}(v))<(n-1, \dots, n-1).
\]
It is clear that $F_n$, $F_{n-1}$ are closed under the action of Steenrod squares since, 
for each monomial $v$ in $P_n^{d}$,
\[
\mathrm{Sq}^a v 
\]
is a linear combination of monomials $w$ such that 
\[
(\omega_0(w), \omega_1(w),\dots ) < (\omega_0(v), \omega_1(v),\dots ).
\]
We denote the projections by the same symbol $\pi\colon P_n \to P_n/F_n$, $\pi\colon P_{n-1}\to P_{n-1}/F_{n-1}$.

Throughout  the rest of this section,  let $k$, $n$ be fixed positive integers
and $d_1=(n-1)(2^k-1)$. 
We consider the quotient spaces
\[
(P_n/F_n)^{d_1}=P_n^{d_1}/(F_n\cap P_n^{d_1}) 
\]
and 
\[
(P_{n-1}/F_{n-1})^{d_1}=P_{n-1}^{d_1}/(F_{n-1}\cap P_{n-1}^{d_1}).
\]
A monomial $v$ in $P_{n-1}^{d_1}$ is $x_1^{2^{k-1}}\cdots x_{n-1}^{2^{k}-1}$ or $\omega(v)<(n-1, \dots, n-1)$, that is, $v \in F_{n-1}$.
So, 
it is clear that $(P_{n-1}/F_{n-1})^{d_1}=\mathbb{Z}/2$ and spanned 
by the single element 
\[
x_1^{2^k-1}\cdots x_{n-1}^{2^k-1}=\pi(x_1^{2^k-1}\cdots x_{n-1}^{2^k-1}).
\]
To describe a basis for the vector space $(P_n/F_n)^{d_1}$, we need the following definitions.
For $1\leq \ell \leq n$, let $\mathrm{Mono}(\ell)$ is the set of monotone increasing functions 
\[
\{ 1, \dots, \ell \} \to \{ 1, \dots, n\}.
\]
We identify $\sigma \in \mathrm{Mono}(\ell)$ with the permutation $\sigma$ on $\{1, \dots, n\}$
such that 
\[
\sigma(1)<\cdots < \sigma(\ell), \quad \sigma(\ell+1)<\cdots <\sigma(n).
\]
Then, the permutation $\sigma$ of $\{1, \dots, n\}$ acts on $P_n$ in the obvious manner, that is, 
\[
\sigma(x_1^{e_1}\cdots x_n^{e_n})=x_{\sigma(1)}^{e_1}\cdots x_{\sigma(n)}^{e_n}.
\]
For an integer $\ell$ such that $1\leq \ell\leq \min\{ k, n\}$, let us define the monomial $v_\ell$ in $P_n^{d_1}$ by 
\[
v_\ell=x_1^{2^{\ell-1}-1} x_2^{2^k-2^{\ell-2}-1} \cdots x_{\ell}^{2^{k}-2^{0}-1} x_{\ell+1}^{2^k-1}\cdots x_n^{2^k-1}.
\]
First, we prove that the set of monomials 
\[
\{ \pi(\sigma(v_\ell)) \; |\; 1\leq \ell\leq \min\{ k, n\}, \sigma \in \mathrm{Mono}(\ell)\}
\]
 spans the vectors space  $(P_n/F_n)^{d_1}$. To this end, we prove the following Propositions~\ref{proposition:3.1} and \ref{proposition:3.2}.


\begin{proposition}\label{proposition:3.1}
Let $v$ be a monomial in $P_n^{d_1}$. Then, 
\[
(\omega_0(v), \dots, \omega_{k-1}(v))\leq (n-1, \dots, n-1).
\]
\end{proposition}

\begin{proof}
Proof by contradiction.
Suppose that 
\[
(\omega_0(v), \dots, \omega_{\ell-1}(v))=(n-1, \dots, n-1)
\]
and 
\[
\omega_\ell(v)=n,
\]
for $1\leq \ell < k$.
Then, on the one hand, since $d_1$ could be written as 
\[
\sum_{j=0}^{\infty} \omega_j(v) 2^{j-1}=(n-1)(2^{\ell}-1)+ n 2^{\ell}+\sum_{j=\ell+1}^{\infty} \omega_j(v) 2^{j-1}, 
\]
we have
\[
d_1-(n-1)(2^{\ell+1}-1)= 2^{\ell}+\omega_{\ell+1}(v) 2^{\ell+1}+\cdots.
\]
It is not divisible by $2^{\ell+1}$.
On the other hand, since $d_1=(n-1)(2^k-1)$, 
we have
\[
d_1-(n-1)(2^{\ell+1}-1)=(n-1)(2^k-2^{\ell+1}).
\]
It is divisible by $2^{\ell+1}$. 
 It is a contradiction. 
\end{proof}


\begin{proposition}\label{proposition:3.2}
For a  monomial $v$ in $P_n^{d_1}$  such that 
\[
(\omega_0(v), \dots, \omega_{k-1}(v))=(n-1,\dots, n-1),
\]
there exists a unique  pair  $(\ell, \sigma)$ such that $1\leq \ell\leq \min\{ k, n\}$, $\sigma\in \mathrm{Mono}(\ell)$
and 
\[
v\equiv \sigma(v_\ell)\mod F_n.
\]
\end{proposition}

Now, we prove Proposition~\ref{proposition:3.2}. 
For each monomial $v$ in $P_n$, let 
\[
u_j(v)=(\alpha_{1j}(v), \dots, \alpha_{nj}(v)).
\]


\begin{lemma}\label{lemma:3.3}
Let $v$ be a monomial such that
\[
(\omega_0(v), \dots, \omega_{k-1}(v))=(n-1, \dots, n-1).
\]
Suppose that $u_j(v)<u_{j+1}(v)$. 
Let $v'$ be the unique monomial such that $u_a(v)=u_a(v')$ for $a\not=j, j+1$ and
$u_j(v')=u_{j+1}(v)$, $u_{j+1}(v')=u_{j}(v)$. Then, 
\[
v\equiv v'\mod F_n.
\]
\end{lemma}

\begin{proof}
Let $w$ be the monomial such that 
\[
u_{j}(w)=(1, \dots, 1),
\]
\[
u_{j+1}(w)=(\alpha_{1,j}(v)\alpha_{1,j+1}(v), \dots, 
\alpha_{n,j}(v)\alpha_{n,j+1}(v)),
\]
and 
\[
u_a(w)=u_a(v)
\]
 for $a \not= j, j+1$.

Let $w_0$, $w_1$ be monomials such that
\[
(\omega_0(w_0), \dots, \omega_{j-1}(w_0), \omega_j(w_0), \omega_{j+1}(w_0), \dots )=(\omega_0(w), \dots, \omega_{j-1}(w), 0, 0, \dots  )
\]
and 
\[
(\omega_0(w_1), \omega_1(w_1), \dots )=(\omega_j(w), \omega_{j+1}(w), \dots).
\]
Then, we have
\[
w=w_0 (w_1)^{2^j} 
\]
and, by the Cartan formula, we have
\[
\sq^{2^j} (w)= \sum_{a+b=2^j}  \sq^a (w_0) \sq^b (w_1^{2^{j}}).
\]
Furthermore, by the Cartan formula, for $0<b<2^j$, we have
\[
\sq^b (w_1^{2^j})=0
\]
and for $b=2^j$, we have 
\[
\sq^b (w_1^{2^j})=(\sq^1 (w_1))^{2^j}.
\]
So,  we have
\[
\sq^{2^j} (w)=w_0   (\sq^1 (w_1))^{2^{j}}+ (\sq^{2^j}(w_0)) w_1^{2^j}.
\]
For $a>0$ and a monomial $w'$, $\sq^a (w')$ is a linear combination of monomials $w''$ such that
\[ 
(\omega_0(w''), \omega_1(w''), \dots )< (\omega_0)w'), \omega_1(w'), \dots).
\]
Hence, we have
\[
\sq^{2^{j}}(w) \equiv  w_0 (\sq^1 (w_1))^{2^j} \equiv v+v' \mod F'_n,
\]
where $F'_n$ is the subspace spanned by monomials $v''$ such that 
\[
(\omega_0(v''), \dots, \omega_{k-1}(v''))<(n-1, \dots, n-1).
\]
Therefore, we have the desired result.
\end{proof}

Thus, for each monomial $v$ in $P_n^{d_1}$ such that 
\[
(\omega_0(v), \dots, \omega_{k-1}(v))=(n-1, \dots, n-1),
\]
there exists a  monomial $v'$ such that 
\[
(\omega_0(v'), \dots, \omega_{k-1}(v'))=(n-1, \dots, n-1),
\]
\[
u_0(v')\geq u_1(v') \cdots \geq u_{k-1}(v')
\]
 and 
 \[
 v\equiv v'\mod F_n.
 \]


\begin{lemma}\label{lemma:3.4}
Let $v$ be a monomial such that
\[
(\omega_0(v), \dots, \omega_{k-1}(v))=(n-1, \dots, n-1).
\]
Suppose that 
\[
u_j(v)=u_{j+1}(v)>u_{j+2}(v).
\]
Let $v'$ be the unique monomial such that $u_a(v)=u_a(v)$ for $a \not=j, j+1, j+2$ and
$u_j(v')>u_{j+1}(v')=u_{j+2}(v')$, $u_{j}(v')=u_{j}(v)$, $u_{j+2}(v')=u_{j+2}(v)$. Then, $v\equiv v'$ modulo $F_n$.
\end{lemma}

\begin{proof}
Let $w$ be the unique  monomial such that 
\[
u_j(w)=u_{j+1}(w)=(1,\dots, 1),
\]
\[
u_{j+2}(w)=(\alpha_{1,j}(v)\alpha_{1, j+2}(v), \dots, \alpha_{n,j}(v)\alpha_{n, j+2}(v)),
\]
and
\[
u_{a}(w)=u_{a}(v)
\]
for $a \not=j, j+1, j+2$.
Let $v''$ be the unique monomial such that 
\[
u_{j}(v'')=u_{j+1}(v'')<u_{j+2}(v''), 
\]
\[
u_{j}(v'')=u_{j+2}(v),  u_{j+2}(v'')=u_{j}(v),
\]
and
\[
u_{a}(v'')=u_a(v)
\]
for $a \not=j, j+1, j+2$. 
Then, as in the proof of Lemma~\ref{lemma:3.3}, we have
\[
\sq^{2^{j}} w\equiv v+v'' \mod F'_n, 
\]
where $F'_n$ is the subspace used in the proof of Lemma~\ref{lemma:3.3}.
Hence, we have 
\[
v\equiv v'' \mod F_n.
\]
By applying Lemma~\ref{lemma:3.3} repeatedly,  we have 
\[
 v''\equiv v' \mod F_n. \qedhere
\]
\end{proof}

Thus, by Lemmas~\ref{lemma:3.3} and \ref{lemma:3.4},  for each monomial $v$ in $P_n^{d_1}$ such that 
\[
(\omega_0(v), \dots, \omega_{k-1}(v))=(n-1, \dots, n-1),
\]
there exists a  monomial $v'$ such that 
\[
(\omega_0(v'), \dots, \omega_{k-1}(v'))=(n-1, \dots, n-1),
\]
\[
u_0(v')> u_1(v') \cdots >u_{\ell}(v')=\cdots =u_{k-1}(v')
\]
 and 
 \[
 v\equiv v'\mod F_n.
 \]
In other words, there exist $1\leq \ell \leq \min\{ k, n\}$ and $\sigma\in \mathrm{Mono}(\ell)$ such that 
\[
v\equiv \sigma(v_{\ell}) \mod F_n,
\]
where $\alpha_{i}(v)<k$ for $i \in \{ \sigma(1), \dots, \sigma(\ell)\}$ and $\alpha_i(v)=k$ for $i \not \in \{\sigma(1), \dots, \sigma(\ell)\}$.

Next, we prove that 
\[
\{ \pi(\sigma(v_\ell)) \; |\; 1\leq \ell\leq \min\{k, n\}, \sigma \in \mathrm{Mono}(\ell)\}
\]
is linearly independent.
Let $\lambda(\sigma)\colon P_n \to P_{n-1}$ be a ring homomorphism defined by 
\begin{align*}
\lambda(\sigma)(x_i)&= x_i &\mbox{for $i<\sigma(\ell)$,}
\\
\lambda(\sigma)(x_i)&= \sigma(x_1)+\cdots + \sigma(x_{\ell-1})  & \mbox{for $i=\sigma(\ell)$,}
\\
\lambda(\sigma)(x_i)&=x_{i-1}  & \mbox{for $i>\sigma(\ell)$.}
\end{align*}
Let us write 
\[
\tilde{u}_j(v)=x_1^{\alpha_{1j}(v)}\cdots x_n^{\alpha_{nj}(v)}.
\]
Then, we have
\[
v=\prod_{j=0}^{\infty} (\tilde{u}_j(v))^{2^j}
.
\]
It is clear that
\[
\lambda(\sigma)(v)=\prod_{j=0}^{\infty} (\lambda(\sigma)(\tilde{u}_j(v)))^{2^j}
\]
and
\[
\lambda(\sigma)(x_1\cdots\widehat{x}_a \cdots x_n)
=\left\{ 
\begin{array}{ll}
x_1\cdots x_{n-1} +\sum v'& \mbox{if $a\in \{ \sigma(1), \dots, \sigma(\ell) \}$,}
\\
\sum v' & \mbox{otherwise,}
\end{array}
\right.
\]
where $x_1\cdots\widehat{x}_a \cdots x_n$ is the monomial of degree $n-1$ obtained from $x_1\cdots x_n$ by removing $x_a$ and 
$\sum v'$ indicates a linear combination of monomials 
$v'$ such that $\omega_0(v')<n-1$.
Therefore, it is easy to see that the following Lemma~\ref{lemma:3.5} holds.


\begin{lemma}\label{lemma:3.5}
Suppose $1\leq \ell, m \leq \min\{k, n\}$, $\sigma\in \mathrm{Mono}({m})$ and $\tau\in \mathrm{Mono}(\ell)$. 
Then, 
\[
\lambda(\sigma)(\tau(v_\ell))\equiv x_1^{2^k-1}\cdots x_{n-1}^{2^k-1} \not \equiv 0 \mod F_n
\]
if and only if 
\[
\{ \sigma(1), \dots, \sigma(m)\} \supseteq \{ \tau(1), \dots, \tau(\ell)\},
\]
where $v_\ell$ is the monomial $v_\ell$ in Proposition~\ref{proposition:3.2}.
\end{lemma}

It follows from Lemma~\ref{lemma:3.5} that the linear map
\[
\lambda \colon (P_n/F_n)^{d_1} \to \prod_{1\leq \ell\leq \min\{ k, n\}} \left( \prod_{\sigma\in \mathrm{Mono}(\ell)} (P_{n-1}/F_{n-1})^{d_1}\right)
\]
sending $\pi(v)$ to
$
( \pi(\lambda(\sigma)(v)) )
$
is an isomorphism.
Thus,  
\[
\{ \pi(\sigma(v_\ell)) \; |\; 1\leq \ell\leq \min\{k, n\}, \sigma \in \mathrm{Mono}(\ell)\}
\]
 is  a basis for $(P_n/F_n)^{d_1}$.


\section{The motivic hit problem}\label{section:4}

In this section, we give the details of motivic hit problem.
Let 
\[
N_n=M_n/(\tau)=\Lambda_{n}(x_1, \dots, x_n) \otimes \mathbb{Z}/2[y_1, \dots, y_n]
\]
and 
\[
\mathcal{A}'=\mathcal{A}^{*,*}/(\tau).
\]
Then, $N_n$ is  an $\mathcal{A}'$-module.
Let $N_n^{d,*}$ be the subspace of $N_n$ spanned by elements of degree $(d, *)$.
From now on, for the sake of simplicity, we say an element is of degree $d$ if its degree is  $(d, *)$.
Let $\mathcal{A}'_{+}$ be the subset of $\mathcal{A}'$ consisting of positive degree elements in $\mathcal{A}'$ and
\[
\mathcal{A}'_{+}(N_n)=\{ ax \;|\; a \in \mathcal{A}'_{+}, \; x \in N_n\}.
\]
Then, it is easy to see that $QM_n^{d,*}$ in Section~\ref{section:1} is isomorphic to
\[
N_n^{d,*}/(\mathcal{A}'_{+}(N_n) \cap N_n^{d,*}).
\]

We consider the counterpart of $F_n$ in $N_n$.
For the sake of notational simplicity, we write $\Lambda_{n}$, $Y_n$ for 
$\Lambda_{n}(x_1, \dots, x_n)$, $\mathbb{Z}/2[y_1, \dots, y_n]$, respectively. 
We denote by $\Lambda_{n}^a$, $Y_n^{2b}$ the subspaces of $\Lambda_{n}$, $Y_n$ 
spanned by  elements of degree $a$, $2b$, respectively.
For a monomial
\[
z=x_1^{\varepsilon_1}\cdots x_n^{\varepsilon_n} y_1^{e_1}\cdots y_n^{e_n}
\]
in $N_n^{d,*}$,
let us define  $\alpha_{ij}(z)\in \{0,1\}$ by
\[
\varepsilon_i+2e_i=\sum_{j=0}^\infty \alpha_{ij}(z) 2^j.
\]
We define non-negative integer $\alpha_i(z)$, $\omega_j(z)$ by 
\begin{align*}
\alpha_i(z)&= \sum_{j=0}^{\infty} \alpha_{ij}(z), 
\\
\omega_j(z)&=\sum_{i=1}^n \alpha_{ij}(z),
\end{align*}
respectively.
Let $\mathcal{P}$ be the subalgebra of $\mathcal{A}'$ generated by reduced power operations $\wp^1, \wp^2, \wp^3, \dots$ of degree $2$, $4$, $6, \dots .$
Let $\mathcal{P}_+$ be the subset of $\mathcal{P}$ consisting of positive degree elements in $\mathcal{P}$.
Let $G_n$ be the subspace spanned by 
\[
\mathcal{P}_{+}(Y_n)= \{ a x \;|\; a\in \mathcal{P}_+, \; x \in Y_n\}
\]
and monomials $z$ such that
\[
(\omega_1(z), \dots, \omega_k(z))<(n-1, \dots, n-1).
\]
Then,  $G_n\subset Y_n$ is the counterpart of $F_n$ in $Y_n$.
The ring isomorphism $\psi\colon Y_n\to P_n$ sending $y_i$ to $x_i$ 
commutes with the action of $\mathcal{P}$, $\mathcal{A}$ in the sense that
$\psi(\wp^c y)= \sq^c \psi(y)$. Thus, it induces an isomorphism
\[
\psi\colon (Y_n/G_n)^{2b} \to (P_n/F_n)^{b}.
\]
We use this isomorphism to identify $ (Y_n/G_n)^{2b}$ with $(P_n/F_n)^{b}$, so that we can apply results
on $(P_n/F_n)^{d_1}$ in Section~\ref{section:3} for $ (Y_n/G_n)^{2d_1}$.

We denote the projection from $\Lambda_{n}^a \otimes Y_n^{2b}$ to $\Lambda_{n}^a \otimes (Y_n/G_n)^{2b}$ by
\[
\pi\colon \Lambda_{n}^a \otimes Y_n^{2b} \to \Lambda_{n}^a \otimes (Y_n/G_n)^{2b}.
\]
Let $H_n$ be the subspace of $N_n$ spanned by 
\[
\mathcal{A}'_{+}(N_n)
\]
and monomials $z$ such that 
\[
(\omega_0(z), \omega_1(z), \dots, \omega_{k}(z))<(k, n-1, \dots, n-1).
\]
Then,  $H_n\subset N_n$ is the counterpart of $F_n\subset P_n$ in $N_n$.
Since $Q_0$ maps $\Lambda_{n}^{a+1} \otimes Y_n^{2(b-1)}$ to 
$\Lambda_{n}^{a} \otimes Y_n^{2b}$ and, for $c>0$, $\wp^c$ acts trivially on $\Lambda_n$, 
we have the following direct sum decomposition
\[
(N_n/H_n)^{d,*} =\bigoplus_{a+2b=d} \Lambda_{n}^a \otimes (Y_n/G_n)^{2b}/\pi(Q_0(\Lambda_{n}^{a+1}\otimes Y_n^{2(b-1)})).
\]

We prove the following proposition using Propositions~\ref{proposition:3.1} and \ref{proposition:3.2}.


\begin{proposition}\label{proposition:4.1}
Suppose that $d=k+2d_1$, $d_1=(n-1)(2^k-1)$, $1\leq k <n$.
For each  monomial $z$ in 
\[
\Lambda_{n}^k \otimes Y_n^{2d_1}, 
\]
$z\in \Lambda_{n}^k \otimes G_n^{2d_1}$ or 
there exist unique
$\ell$ in $\{ 1, \dots, k\}$, 
$\sigma_1\in \mathrm{Mono}(k)$ and  $\sigma_2\in \mathrm{Mono}(\ell)$ such that
\[
z\equiv 
 \sigma_1(x_1\cdots x_k) \sigma_2(\psi^{-1}(v_\ell)) \mod \Lambda_{n}^k \otimes G_n^{2d_1},
\]
where $v_\ell$ is the monomial $v_\ell \in P^{d_1}_n$  in Proposition~\ref{proposition:3.2}.
\end{proposition}

\begin{proof}
Suppose that  $z=x_{i_1} \cdots x_{i_k} \otimes v$ and that $z$ is not in $\Lambda^k \otimes G_n^{2d_1}$.
Then
by Proposition~\ref{proposition:3.1}, we have 
\[
(\omega_0(\psi(v)), \dots, \omega_{k-1}(\psi(v)))= (n-1, \dots, n-1).
\]
So, by Proposition~\ref{proposition:3.2}, there exists the unique pair $(\ell, \sigma_2)$ such that 
\[
\psi(v) \equiv \sigma_2(v_\ell) \mod F_n
\]
in $P_n^{d_1}$.
Let us define $\sigma_1\in \mathrm{Mono}(k)$ by $\sigma_1(j)=i_j$. 
Then, we have
\[
z\equiv 
 \sigma_1(x_1\cdots x_k) \sigma_2(\psi^{-1}(v_\ell)) \mod \Lambda_{n}^k \otimes G_n^{2d_1},
\]
as desired.
\end{proof}

Let $\mathcal{M}_0$ be the set of monomials $z=\sigma_1(x_1\cdots x_k) \sigma_2(\psi^{-1}(v_\ell))$ in 
$\Lambda_{n}^k\otimes Y_n^{2d_1}$ such that $\alpha_i(z)<k$ for some $i$
and $\mathcal{M}_1$ the set of monomials $z=\sigma_1(x_1\cdots x_k) \sigma_2(\psi^{-1}(v_\ell))$ in 
$\Lambda_{n}^k \otimes Y_n^{2d_1}$ such that $\alpha_i(z)=k$ for all $i$.
In Section~\ref{section:3}, we proved that
\[
\{ \pi (\sigma_2 (v_\ell)) \;|\; 1\leq \ell \leq k, \sigma_2 \in \mathrm{Mono}(\ell)\}
\]
 is a basis for $(P_n/F_n)^{d_1}$, 
Since 
\[
\{ \sigma_1(x_1\cdots x_k)\; |\; \sigma_1 \in \mathrm{Mono}(k)\}
\] is a basis for $\Lambda_n^k$, 
$\pi(\mathcal{M}_0 \cup \mathcal{M}_1)$ is a basis for $\Lambda_{n}^k \otimes (Y_n/G_n)^{2d_1}$.
If $\ell \not=k$, then 
$\alpha_i(z)<k$ for some $i$ in $\{ 1, \dots, n\}$.
If $\ell=k$ and $\sigma_1\not=\sigma_2$, then 
$\alpha_i(z)<k$ for some $i$ in $\{ 1, \dots, n\}$.
If $\ell=k$ and $\sigma_1=\sigma_2$, then 
\[
v_k=x_1^{2^k-2^{k-1}-1}\cdots x_{k}^{2^k-2^0-1}x_{k+1}^{2^k-1}\cdots x_n^{2^k-1}
\]
and so
\[
x_1\cdots x_k \psi^{-1}(v_k)=
x_1\cdots x_k y_1^{2^k-2^{k-1}-1}\cdots y_k^{2^{k}-2^0-1} y_{k+1}^{2^k-1}\cdots y_n^{2^k-1}
\]
is the monomial $z_k$ in Proposition~\ref{proposition:1.4}.
Therefore, we have
\begin{align*}
\mathcal{M}_0 &=\{ \sigma_1(x_1\cdots x_k) \sigma_2(\psi^{-1}(v_\ell)) \:|\; 1\leq \ell < k, \sigma_1\in\mathrm{Mono}(k), \sigma_2\in \mathrm{Mono}(\ell)\} 
\\
& \cup 
\{ \sigma_1(x_1\cdots x_k) \sigma_2(\psi^{-1}(v_k)) \:|\; \sigma_1, \sigma_2 \in\mathrm{Mono}(k), \sigma_1\not=\sigma_2
\}, 
\\
\mathcal{M}_1 &=\{ \sigma_1(x_1\cdots x_k) \sigma_2(\psi^{-1}(v_k)) \:|\; \sigma_1, \sigma_2 \in\mathrm{Mono}(k), \sigma_1 =\sigma_2
\}
\\
&=\{ \sigma(z_k) \;|\; \sigma\in \mathrm{Mono}(k)\}.
\end{align*}

\section{Proof of Proposition~\ref{proposition:1.4}}\label{section:5}

Throughout  this section, we suppose that $d=k+2d_1$, $d_1=(n-1)(2^k-1)$, $1\leq k <n$.
The algebra $\mathcal{A}'$ 
 is generated by $Q_0$ and the reduced power operations $\wp^1$, $\wp^2$, $\dots$.
By definition, we have $\mathcal{P}_+(Y_n)\subset G_n$.
Moreover,  $\wp^a(x)=0$ in $N_n$ for $a>0$, $x\in \Lambda_n$.
Therefore, we have
\[
\pi (\wp^a(x\otimes y))=\pi( x \otimes \wp^a(y))=0
\]
 for each monomial $x\otimes y$ in $N_n^{d-2a,*}$.
So, we prove Proposition~\ref{proposition:1.4} by proving the following proposition.


\begin{proposition}\label{proposition:5.1}
For each monomial  $z$ in $N_n^{d-1}$ with $\omega_0(z)=k+1$, 
\[
\pi(Q_0 (z))\in \Lambda_{n}^k \otimes (Y_n/G_n)^{2d_1}
\]
 is a linear combination of $\pi(z')$ {\rm (}$z' \in \mathcal{M}_0${\rm )} and
 $\pi(\sigma_1(z_k)+\sigma_2(z_k))$, where $\sigma_1, \sigma_2 \in \mathrm{Mono}(k)$ and $z_k$ is the monomial
 $z_k$  in Proposition~\ref{proposition:1.4}.
\end{proposition}

From Propositions~\ref{proposition:4.1} and \ref{proposition:5.1}, we have that $\pi(z_k)$ is not  in 
\begin{align*}
&\pi(Q_0(\Lambda_{n}^{k+1} \otimes Y_n^{2(d_1-1)}))
\\
=\; & \pi(Q_0(\Lambda_{n}^{k+1} \otimes Y_n^{2(d_1-1)}))+
\sum_{a>0} \pi(\wp^a(\Lambda_{n}^{k} \otimes Y_n^{2(d_1-a)})
\\
=\; &\pi (\mathcal{A}_+'(N_n))
\end{align*}
Thus, once we prove Proposition~\ref{proposition:5.1}, we complete the proof of Proposition~\ref{proposition:1.4}.

Let $z$ be a monomial in $\Lambda_n^{k+1}\otimes Y_n^{2(d_1-1)}$. 

If $\omega_1(z)=n-1$, then 
\begin{align*}
\omega_2(z) \cdot 2^2+\omega_3(z)\cdot 2^3+\cdots &= d-1-(k+1)-2(n-1)
\\
&=2d_1 -2n
\\
&=2(n-1) (2^k-1)-2n
\\
&= (n-1)2^{k+1}+2.
\end{align*}
Since $k\geq 1$,  $(n-1)2^{k+1}+2$ is not divisible by $4$ but 
$\omega_2(z) \cdot 2^2+\omega_3(z)\cdot 2^3+\cdots$ is divisible by $4$. It is a contradiction. So, 
$\omega_1(z) \not=n-1$. 

If $\omega_1(z)<n-2$, then $Q_0(z)$ is a linear combination of monomials $z'$ such that $\omega_0(z')=k$ and 
$\omega_1(z')<n-1$.
Hence, $Q_0(z) \in \Lambda_{n}^k \otimes G_n^{2d_1}$.

There remain two cases: $\omega_1(z)=n$ or $n-2$.
First, we deal with the case $\omega_1(z)=n$.


\begin{proposition}\label{proposition:5.2}
Suppose that $z$ is a monomial in $N_n^{d-1}$ such that
\[
(\omega_0(z), \omega_1(z))=(k+1,n).
\]
Then, there exist  monomials $z'$ such that 
\[
(\omega_0(z'), \omega_1(z'))=(k+1, n-2)
\]
and 
\[
Q_0(z)=Q_0(\sum z').
\]
\end{proposition}
\begin{proof}
Without loss of generality, we may assume that 
\begin{align*}
u_0(z)&=(1, \dots, 1, 0, 0, \dots, 0),
\\
u_1(z)&=(1, \dots, 1, 1, 1, \dots, 1).
\end{align*}
Let $z''$ be a monomial such that 
\begin{align*}
u_0(z'')&=(1, \dots, 1, 1, 0, \dots, 0), 
\\
u_1(z'')&=(1, \dots, 1, 0, 1, \dots, 1)
\end{align*}
and $u_a(z'')=u_a(z)$ for $a \geq 2$.
Then, 
\[
Q_0(z'')=z+\sum z'
\]
where $\sum z'$  is a linear combination of monomials $z'$ such that $(\omega_0(z'), \omega_1(z'))=(k+1, n-2)$.
Since $Q_0Q_0= 0$, we have that
\[
Q_0(z)=Q_0(\sum z')
\]
as desired.
\end{proof}

So, Proposition~\ref{proposition:5.3} below completes the proof of Proposition~\ref{proposition:5.1}.


\begin{proposition}\label{proposition:5.3}
Suppose that $z$ is a monomial in $N_n^{d-1}$ such that
\[
(\omega_0(z), \omega_1(z))=(k+1,n-2).
\]
If $\alpha_i(z)<k$ for some $i$, then 
$Q_0(z)$ is congruent to a linear combination of monomials in $\mathcal{M}_0$ modulo $\Lambda_n^{k} \otimes G_n^{2d_1}$.
If  $\alpha_i(z)=k$ for all $i \in \{1, \dots, n\}$, then 
\[
Q_0(z)\equiv \sigma_1(z_k)+\sigma_2(z_k) \mod \Lambda_{n}^k \otimes G_n^{2d_1}.
\]
\end{proposition}

\begin{proof}
In what follows, we consider everything modulo $\Lambda_n^{k} \otimes G_{n}^{2d_1}$.
The element $Q_0(z)$ is congruent to a linear combination of monomials $z'$ such that 
$\alpha_i(z')\leq \alpha_i(z)$ for all $i \in \{1, \dots, n\}$. 
Hence, if $\alpha_i(z)<k$ for some $i$, then $Q_0(z)$ is congruent to a linear combination of  $z'$ 
such that $\alpha_i(z')<k$. 
Hence, it is congruent to a linear combination of elements in $\mathcal{M}_0$.

If $\omega_1(z)=n-2$ and $\alpha_i(z)=k$ for all $i \in \{1, \dots, n\}$, 
then without loss of generality, we may assume that 
\begin{align*}
u_0(z)&=(1,1, 1, \dots , 1, 0, \dots, 0),
\\
u_1(z)&=(0,0,1, \dots, 1, 1, \dots,  1).
\end{align*}
Then, $Q_0(z)$  is congruent to $\sigma_1(z_k)+\sigma_2(z_k)$ 
where $\sigma_1, \sigma_2$ in $\mathrm{Mono}(k)$, $\sigma_1(1)=1$, $\sigma_1(2)=3, \dots, \sigma_1(k)=k+1$, 
$\sigma_2(1)=2, \dots, \sigma_2(k)=k+1$.
\end{proof}



\begin{bibdiv}



\begin{biblist}

\bib{ault-2014}{article}{
   author={Ault, Shaun},
   title={Bott periodicity in the hit problem},
   journal={Math. Proc. Cambridge Philos. Soc.},
   volume={156},
   date={2014},
   number={3},
   pages={545--554},
   issn={0305-0041},
   review={\MR{3181639}},
   doi={10.1017/S0305004114000085},
}

\bib{dugger-isaksen-2010}{article}{
   author={Dugger, Daniel},
   author={Isaksen, Daniel C.},
   title={The motivic Adams spectral sequence},
   journal={Geom. Topol.},
   volume={14},
   date={2010},
   number={2},
   pages={967--1014},
   issn={1465-3060},
   review={\MR{2629898}},
   doi={10.2140/gt.2010.14.967},
}

\bib{minami-1995}{article}{
   author={Minami, Norihiko},
   title={The Adams spectral sequence and the triple transfer},
   journal={Amer. J. Math.},
   volume={117},
   date={1995},
   number={4},
   pages={965--985},
   issn={0002-9327},
   review={\MR{1342837}},
   doi={10.2307/2374955},
}

\bib{pengelley-williams-2015}{article}{
   author={Pengelley, David},
   author={Williams, Frank},
   title={Sparseness for the symmetric hit problem at all primes},
   journal={Math. Proc. Cambridge Philos. Soc.},
   volume={158},
   date={2015},
   number={2},
   pages={269--274},
   issn={0305-0041},
   review={\MR{3310245}},
   doi={10.1017/S0305004114000668},
}

\bib{singer-1989}{article}{
   author={Singer, William M.},
   title={The transfer in homological algebra},
   journal={Math. Z.},
   volume={202},
   date={1989},
   number={4},
   pages={493--523},
   issn={0025-5874},
   review={\MR{1022818}},
   doi={10.1007/BF01221587},
}

\bib{sum-2015}{article}{
   author={Sum, Nguyen},
   title={On the Peterson hit problem},
   journal={Adv. Math.},
   volume={274},
   date={2015},
   pages={432--489},
   issn={0001-8708},
   review={\MR{3318156}},
   doi={10.1016/j.aim.2015.01.010},
}

\bib{voevodsky-2003}{article}{
   author={Voevodsky, Vladimir},
   title={Reduced power operations in motivic cohomology},
   journal={Publ. Math. Inst. Hautes \'Etudes Sci.},
   number={98},
   date={2003},
   pages={1--57},
   issn={0073-8301},
   review={\MR{2031198}},
   doi={10.1007/s10240-003-0009-z},
}

\bib{walker-wood-2018}{book}{
   author={Walker, Grant},
   author={Wood, Reginald M. W.},
   title={Polynomials and the ${\rm mod}\, 2$ Steenrod algebra. Vol. 1. The
   Peterson hit problem},
   series={London Mathematical Society Lecture Note Series},
   volume={441},
   publisher={Cambridge University Press, Cambridge},
   date={2018},
   pages={xxiv+346},
   isbn={978-1-108-41448-7},
   isbn={978-1-108-41406-7},
   review={\MR{3729477}},
}

\bib{wood-1989}{article}{
   author={Wood, R. M. W.},
   title={Steenrod squares of polynomials and the Peterson conjecture},
   journal={Math. Proc. Cambridge Philos. Soc.},
   volume={105},
   date={1989},
   number={2},
   pages={307--309},
   issn={0305-0041},
   review={\MR{974986}},
   doi={10.1017/S0305004100067797},
}

\bib{yagita-2005}{article}{
   author={Yagita, Nobuaki},
   title={Applications of Atiyah-Hirzebruch spectral sequences for motivic
   cobordism},
   journal={Proc. London Math. Soc. (3)},
   volume={90},
   date={2005},
   number={3},
   pages={783--816},
   issn={0024-6115},
   review={\MR{2137831}},
   doi={10.1112/S0024611504015084},
}

	
	\end{biblist}
	
	\end{bibdiv}
\end{document}